\documentclass[reqno,12pt]{amsart}
\usepackage{amsmath,amsfonts,amsthm,amssymb,indentfirst}
\usepackage[all,poly]{xy} 
\usepackage{color}
\usepackage[pagebackref,colorlinks]{hyperref}
\usepackage{kantlipsum}
\usepackage[all,poly]{xy}
\usepackage{mathrsfs}
\usepackage{enumerate}

\theoremstyle{plain}
\newtheorem{lemma}{Lemma}[section] 
\newtheorem{theorem}[lemma]{Theorem}

\newtheorem{proposition}[lemma]{Proposition}

\theoremstyle{definition}
\newtheorem{remark}[lemma]{Remark}
\newtheorem{example}[lemma]{Example}

\newcommand\Zset{\mathbb {Z}}

\newcommand{\so}{\mathbf{s}}
\newcommand{\ra}{\mathbf{r}}

\newcommand{\ann}{\operatorname{ann}}
\newcommand{\gr}{\operatorname{gr}}
\newcommand{\ol}{\overline}

\begin{document}
\title{Graph characterization of the annihilator ideals of Leavitt path algebras}

\author{Lia Va\v s}
\address{Department of Mathematics, Saint Joseph's University, Philadelphia, PA 19131, USA}
\email{lvas@sju.edu}



\begin{abstract}
If $E$ is a graph and $K$ is a field, we consider an ideal $I$ of the Leavitt path algebra $L_K(E)$ of $E$ over $K$. We describe the admissible pair corresponding to the smallest graded ideal which contains $I$ where the grading in question is the natural grading of $L_K(E)$ by $\mathbb Z$. Using this description, 
we show that the right and the left annihilators of $I$ are {\em equal} (which can be somewhat surprising given that $I$ may not be self-adjoint). In particular, we establish that both  annihilators correspond to the same admissible pair and its description produces the characterization from the title. Then, we turn to the property that the right (equivalently left) annihilator of any ideal is a direct summand and recall that a unital ring with this property is said to be quasi-Baer. We
exhibit a condition on $E$ which is equivalent to unital $L_K(E)$ having this property. 
\end{abstract}

\maketitle

\section{Introduction}

If $R$ is a ring (associative but not necessarily unital) and $M$ is a left $R$-module, then $\ann_l(M)=\{r\in R\mid rm=0$ for all $m\in M\}$ is a two-sided ideal of $R$ called the {\em left annihilator} of $M.$ Similarly, if $N$ is a right $R$-module, then the ideal $\ann_r(N)=\{r\in R\mid nr=0$ for all $n\in N\}$ is called the {\em right annihilator} of $N.$ If $B$ is both a left and a right $R$-module, $\ann(B)=\{r\in R\mid rb=br=0\mbox{ for any }b\in B\}$ is the {\em annihilator} of $B.$

The ideals which are annihilators of other ideals have been called {\em annihilator ideals}. Annihilator ideals of a Leavitt path algebra have recently been studied in \cite{Malazani}, \cite{Goncalves_Royer_regular_ideals} and \cite{Lia_annihilators}.
If $E$ is a graph and $K$ is a field and if the Leavitt path algebra $L_K(E)$ is considered naturally graded by the group of integers, then  $\ann(I)$ is graded for any ideal $I$ (not necessarily graded) of $L_K(E)$ 
by \cite[Theorem 3.3]{Goncalves_Royer_regular_ideals}. By \cite[Proposition 3.1]{Lia_annihilators}, the same holds for $\ann_l(I)$ and $\ann_r(I).$ However, while these results establish that $\ann_l(I),$ $\ann_r(I),$ and $\ann(I)$ are graded, the exact relations between these three graded ideals have not yet been established. 

In this paper, we establish such exact relations by describing the corresponding admissible pairs of the three annihilators. Namely, any graded ideal of $L_K(E)$ is uniquely determined by a pair of two sets of vertices of the graph $E$, known as an {\em admissible pair}. 
In Theorem  \ref{theorem_annihilators}, we show that all three annihilators $\ann_l(I),$ $\ann_r(I),$ and $\ann(I)$ correspond to the {\em same} admissible pair. So, the three annihilators are {\em equal}. On one hand, this may not be that surprising knowing that $L_K(E)$ is an involutive algebra and, as such, has certain left-right symmetry. On the other hand, this may be surprising given that an ideal of $L_K(E)$ is not necessarily self-adjoint.  

To prove Theorem \ref{theorem_annihilators}, we consider the smallest graded ideal $I^{\gr}$ which contains an ideal $I.$ 
In Theorem \ref{theorem_graded_envelope} (which can be relevant in its own right), we describe $I^{\gr}$ in terms of its admissible pair. Then, we use \cite[Proposition 3.5]{Lia_annihilators} which exhibits the admissible pair of the annihilator of a graded ideal. Using this result, we prove Theorem \ref{theorem_annihilators} stating that if $I^{\gr}=I(H,S),$ if $H^\bot$ is the set of vertices which do not emit paths to $H,$ and if $B_{H^\bot}$ is the set of breaking vertices of $H^\bot$ (we review this concept in section \ref{subsection_ideals_prerequisites}) then 
\[\ann_l(I)=\ann_r(I)=\ann(I)=\ann(I^{\gr})=I(H^\bot, B_{H^\bot}).\]

In section \ref{section_quasi_baer}, we turn to the ring-theoretic condition that the left (equivalently right) annihilator of an ideal is a direct summand.  If $R$ is a unital ring, this condition is equivalent with the requirement that, for any right ideal $I,$ there is an idempotent $\varepsilon\in R$ such that $\ann_r(I)=\varepsilon R$ and a ring which satisfies this property is said to be right {\em quasi-Baer}. In \cite[Lemma 1]{Clark}, it is shown that this definition is left-right symmetric and that an equivalent statement is obtained by requiring $I$ to be a double-sided ideal. Because of this, the left-right specification in front of ``quasi-Baer'' can be dropped. If $R$ has this property, then $R$ is unital (the identity is an idempotent obtained for $I=0$).

We generalize this concept to graded rings and establish the relation between quasi-Baer and graded quasi-Baer properties in Proposition \ref{proposition_graded_vs_nongraded_quasi_Baer}. Turning to Leavitt path algebras, we characterize when $L_K(E)$ is (graded) quasi-Baer in terms of the conditions on $E$ (Proposition \ref{proposition_quasi_baer}). Considering $L_K(E)$ as an involutive ring, we show that these conditions on $E$ also characterize when $L_K(E)$ is a (graded) quasi-Baer $*$-ring (Proposition \ref{proposition_star_qB}). 

\section{Prerequisites and preliminaries}

\subsection{Graded rings}
A ring $R$ (not necessarily unital) is {\em graded} by a group $\Gamma$ if $R=\bigoplus_{\gamma\in\Gamma} R_\gamma$ for additive subgroups $R_\gamma$ and $R_\gamma R_\delta\subseteq R_{\gamma\delta}$ for all $\gamma,\delta\in\Gamma.$ The elements $\bigcup_{\gamma\in\Gamma} R_\gamma$ are {\em homogeneous}. A left (right, double-sided) ideal $I$ of $R$ is {\em graded} if $I=\bigoplus_{\gamma\in \Gamma} I\cap R_\gamma.$ 

\subsection{Graphs and Leavitt path algebras}
If $E$ is a directed graph, we let $E^0$ denote the set of vertices, $E^1$ denote the set of edges, and $\so$ and $\ra$ denote the source and the range maps of $E.$ We adopt the standard definitions of a sink, an infinite emitter, a regular vertex, a finite graph, a path in a graph, and a cycle of $E$ (see \cite{LPA_book} for any of those). For $V\subseteq E^0,$
the {\em root} $R(V)$ of $V$ is the set of vertices $v\in E^0$ such that $\so(p)=v$ and $\ra(p)\in V$ for some path $p.$   

We also adopt the standard definition of the Leavitt path algebra $L_K(E)$ of a graph $E$ over a field $K$ (\cite[Definition 1.2.3]{LPA_book}). Recall that a ring $R$ is {\em locally unital} if for every finite set $F\subseteq R,$ there is an idempotent $\varepsilon\in R$ such that $F\subseteq \varepsilon R\varepsilon.$ The algebra $L_K(E)$ is locally unital (with the finite sums of vertices as the local units). The algebra $L_K(E)$ is unital if and only if $E^0$ is finite in which case the sum of all vertices is the identity.
If we consider $K$ to be trivially graded by $\Zset,$ $L_K(E)$ is naturally graded by $\Zset$ so that the $n$-component $L_K(E)_n$ is the $K$-linear span of the elements $pq^\ast$ for paths $p, q$ with $|p|-|q|=n$ where $|p|$ denotes the length of a path $p.$ 

\subsection{Ideals and graded ideals of a Leavitt path algebra} \label{subsection_ideals_prerequisites}

A subset $H$ of $E^0$ is said to be {\em hereditary} if $\ra(p)\in H$ for any path $p$ such that $\so(p)\in H.$ The set $H$ is {\em saturated} if $v\in H$ for any regular vertex $v$ such that $\ra(\so^{-1}(v))\subseteq H.$ If $H$ is hereditary, let  
\[B_H=\{v\in E^0-H\,|\, v\mbox{ is an infinite emitter and }\so^{-1}(v)\cap \ra^{-1}(E^0-H)\mbox{ is nonempty and finite}\}\]
and, for $v\in B_H,$ let $v^H$ stand for $v-\sum ee^*$ where the sum is taken over $e\in \so^{-1}(v)\cap \ra^{-1}(E^0-H).$ 

An {\em admissible pair} is a pair $(H, S)$ where $H\subseteq E^0$ is hereditary and saturated and $S\subseteq B_H.$ 
For an admissible pair $(H,S)$, if $S^H$ denotes the set $\{v^H\mid v\in S\}$ and  $I(H,S)$ denotes the ideal generated by the elements $H\cup S^H,$ then $I(H,S)$   is graded and it is the $K$-linear span of the elements $pq^*$ for paths $p,q$ with $\ra(p)=\ra(q)\in H$ and the elements $pv^Hq^*$ for paths $p,q$ with $\ra(p)=\ra(q)=v\in S$ (see \cite[Theorem 2.4.8]{LPA_book}). The converse holds as well: for a graded ideal $I$, the set $H=I\cap E^0$ is hereditary and saturated and, if $S=\{v\in B_H\mid v^H\in I\},$ then  $I=I(H,S)$ (\cite[Theorem 2.5.8]{LPA_book}). The
lattice of graded ideals is isomorphic to the lattice of admissible pairs with the relation 
\[(H,S)\leq (K, T) \mbox{ if }H\subseteq K\mbox{ and }S\subseteq K\cup T.\] 

If $I$ is an ideal (not necessarily graded), it is uniquely determined by an admissible pair $(H,S)$ where  $H=I\cap E^0$ and $S=\{v\in B_H\mid v^H\in I\},$ by a set $C$ contained in the set $C_H$ of cycles with vertices outside of $H$ such that every exit from $c\in C_H$ has the range in $H,$ and by a sets $P$ contained in the set of non-constant polynomials in $K[x]$ with the $0$-th coefficient $1_K$ (see \cite[Theorem 2.8.10]{LPA_book}). In this case, we write that $I=I((H,S), C, P).$
Such sets $C$ and $P$ determine the set $P_C$ of  elements $p(c)$ for $c\in C$ and $p\in P.$ 
By \cite[Proposition 2.8.5 and Theorem 2.8.10]{LPA_book}, if $I=I((H,S), C, P)$ is an ideal of $L_K(E),$
then $I$ is generated by  $H\cup S^H\cup P_C.$

\subsection{Annihilator ideals} 
If $I$ is an ideal of $L_K(E),$ then $\ann_l(I),$ $\ann_r(I),$ and $\ann(I)$ are graded ideals of $L_K(E)$ by \cite[Theorem 3.3]{Goncalves_Royer_regular_ideals} and \cite[Proposition 3.1]{Lia_annihilators}. By \cite[Corollary 3.3 and Proposition 3.5]{Lia_annihilators}, if $I(H,S)$ is a graded ideal of $L_K(E)$ and if we let $H^\bot=E^0-R(H)$ and $S^\bot=B_{H^\bot}-S,$ then 
\[\ann_l(I(H,S))=\ann_r(I(H,S))=\ann(I(H,S))=I(H^\bot, S^\bot).\]
\begin{remark}
We claim that $S^\bot=B_{H^\bot}.$ Indeed, if  $v\in B_{H^\bot},$ then $v$ emits infinitely many edges to $H^\bot$ and nonzero and finitely many to $E^0-H^\bot=E^0-(E^0-R(H))=R(H).$ So, $v$ emits only finitely many edges to $H.$ Hence, $v\notin B_H,$ so $v\notin S.$
This shows that 
\[\ann(I(H,S))=I(H^\bot, B_{H^\bot})\]
so $\ann(I(H,S))$ does not depend on $S.$ 
\label{remark_on_S_bot} 
\end{remark}

In \cite{Lia_annihilators}, an admissible pair $(H,S)$ is said to be {\em reflexive} if $(H,S)=(H^{\bot\bot}, S^{\bot\bot}).$ 
By \cite[Proposition 3.10]{Lia_annihilators}, a graded ideal $I=I(H,S)$ is an annihilator ideal if and only if $(H,S)$ is reflexive and, if so, then $S=B_H.$ The above remark also implies this last formula.   

\subsection{ \texorpdfstring{$S$}{TEXT}-saturation of a hereditary set}
\label{subsection_saturation}
If $H$ is a hereditary set,  $S$ is any set of vertices, and $G$ is any hereditary set which contains $H$, let $G^S$ denote the set 
$\{v\in S\mid \ra(\so^{-1}(v))\subseteq G\}.$ The {\em $S$-saturation} $\ol H^S$ of $H$ is  the smallest hereditary and saturated set $G$ which contains $H$ and such that $G^S\subseteq G.$ Note that such smallest set $G$ exists 
since the intersection of the hereditary and saturated sets $G'$ which contain $H$ and satisfy the relation $(G')^S\subseteq G'$ retains all these properties of $G'.$ 
The set $\ol H^S$ can also be defined by an iterative process: if $\Lambda_0^S(H)=H$ and 
\[\Lambda_{n+1}^S(H)=\Lambda^S_n(H)\cup\{v\in E^0-\Lambda_n^S(H)\mid v \mbox{ is either regular or in }S\mbox{ and }\ra(\so^{-1}(v))\subseteq \Lambda_n^S(H)\},\] then $\ol H^S=\bigcup_{n=0}^\infty \Lambda_n^S(H).$ The inclusion $\subseteq$ holds since the union $U=\bigcup_{n=0}^\infty \Lambda_n^S(H)$ is a hereditary and saturated set which contains $H$ and such that $U^S\subseteq U.$ The converse inclusion holds since the induction can be used to show $\Lambda_n^S(H)\subseteq \ol H^S.$ 

When introducing $\ol H^S$ in \cite[Definition 2.5.5]{LPA_book}, it is assumed  that $S\subseteq H\cup B_H.$ We do not require this condition to hold because in cases when $S\nsubseteq H\cup B_H,$ we would still like to have the $S$-saturation of $H$ defined.
For example, if $E$ is the graph below (the symbol $(\infty)$ above an edge $e$ denotes infinitely many edges from $\so(e)$ to $\ra(e)$)
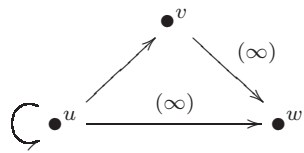
\[\xymatrix{&\bullet^v\ar[dr]^{(\infty)}&\\\bullet^u\ar@(ul, dl)\ar[ur]\ar[rr]^{(\infty)}&&\bullet^w}\]
and if $H=\{w\},$ then $B_H=\{u\}.$ For $S=\{v\},$  
$S\nsubseteq H\cup B_H=\{u,w\}$ and $H^S=\{v\}\nsubseteq H.$ The set $G=\{v,w\}$ is a hereditary and saturated set which contains $H,$ $G^S=\{v\}\subseteq G,$ and $G$ is the smallest such set, so $\ol H^S=G.$

\section{Graded envelope and annihilator ideals}

For the rest of the paper, we fix a graph $E$ and a field $K.$ For an ideal $I$ (not necessarily graded) of the Leavitt path algebra $L_K(E)$ of $E$ over $K$, it is known that there is the largest graded ideal $I_{\gr}$ contained in $I.$ If $H=I\cap E^0$ and $S=\{v\in B_H\mid v^H\in I,\}$ then $I_{\gr}=I(H,S)$ (see \cite[Lemma 2.8.9]{LPA_book}). We consider  the dual concept: the smallest graded ideal $I^{\gr}$ which contains $I.$ Such an ideal exists since the intersection of all graded ideals which contain $I$ is a graded ideal and the term {\em graded envelope} of $I$ would be suitable for $I^{\gr}$. 
The graded envelope 
can also be characterized as the ideal generated by the homogeneous components of the elements of $I,$ i.e. by the set $A=\{a_n\mid a\in I, n\in \Zset\}.$ Indeed, the ideal $I(A)$ generated by $A$ is graded because $A$ consists of  homogeneous elements. The  ideal $I(A)$ contains the elements of $I$ and it is the smallest such ideal since if $ J$ is graded and contains $I,$ 
then any $a\in I$ is such that $a_n\in J,$ for all $n\in\Zset,$ so $A\subseteq J.$ 

We describe the graded envelope in terms of the corresponding admissible pair next.

\begin{theorem}
Let $I$ be an ideal of $L_K(E)$ and let $(H,S),$ $C,$ and $P$ be such that $I=I((H,S), C, P).$ Let $C^0$ be the set of vertices on cycles which are in $C$ and let 
\[G=\ol{H\cup C^0}^{S}\;\;\mbox{ and }\;\;T=S-G.\]
Then $T\subseteq B_G,$ $(H,S)\leq (G,T),$ and
\[I^{\gr}=I(G,T).\]
\label{theorem_graded_envelope}
\end{theorem}
\begin{proof}
If $v\in T=S-G,$ then $v\in B_H,$ so $v$ is an infinite emitter emitting infinitely many edges to $H$ and nonzero and finitely many to $E^0-H.$ So, $v$ emits infinitely many edges to $G$ and finitely many, say $n$, edges to $E^0-G.$ For $v\in B_G$ to hold, we need to show that $n>0.$  Assume, on the contrary, that $n=0.$ Then $\ra(\so^{-1}(v))\subseteq G$ so that $v$ is in $G^S.$ As $G^S\subseteq G,$ $v\in G.$ This contradicts the assumption that $v\in T=S-G.$ Hence, $v\in B_G.$

By the definition of $G$ and $T,$ $H\subseteq G$ and $S\subseteq G\cup T.$ So, $(H,S)\leq (G,T).$  

For the inclusion $I^{\gr}\subseteq I(G,T),$ it is sufficient to prove that $I\subseteq I(G,T).$ As $I$ is generated by $H\cup S^H\cup P_C,$ it is sufficient to prove that $H,$ $S^H,$ and $P_C$ are  contained in $I(G,T).$
Since $(H,S)\leq (G,T),$ $H\cup S^H\subseteq I(G,T).$ 
For $c\in C,$ $\so(c)\in C^0\subseteq G,$ so $c\in I(G,T)$ which implies that $p(c)\in I(G,T)$ for any $p\in P.$

For the inclusion $I(G,T)\subseteq I^{\gr},$ it is sufficient to show that $G\subseteq I^{\gr}$ and that $T^G\subseteq I^{\gr}.$ Let $\Lambda_n, n=0,1,\ldots,$ be the sets $\Lambda_n^S(H\cup C^0)$ defined as in section \ref{subsection_saturation} which have the union equal to $G$. We use induction to show that $\Lambda_n\subseteq I^{\gr}$ for any $n.$ For $n=0,$ we show that $H\cup C^0\subseteq I^{\gr}.$ If $v\in H,$ then  $v\in I$ by the definition of $H,$ so $v\in I^{\gr}.$ If $v\in C^0,$ then there is $c\in C$ and $p\in P$ such that $v$ is a vertex in $c$ and $p(c)\in I.$ 
As the $0$-component of $p(c)$ is $\so(c),$ $\so(c)\in I^{\gr}.$ If $p$ is a part of $c$ from $\so(c)$ to $v,$ then  $p=\so(c)p\in I^{\gr}$ which implies that $v=\ra(p)=p^*p\in I^{\gr}.$

Assuming that $\Lambda_n\subseteq I^{\gr},$ let $v\in \Lambda_{n+1}.$ If $v\in \Lambda_n,$ then $v\in I^{\gr}.$ If $v\in \Lambda_{n+1}-\Lambda_n,$ then $\ra(\so^{-1}(v))\subseteq \Lambda_n,$ so $\ra(\so^{-1}(v))\subseteq I^{\gr}$ by the induction hypothesis. Thus, for any $e\in \so^{-1}(v),$   $ee^*=e\ra(e)e^*\in I^{\gr}.$ If $v$ is regular, then $v=\sum_{e\in\so^{-1}(v)} ee^*\in I^{\gr}.$ If $v$ is in $S,$ then $v^H\in I\subseteq I^{\gr}.$ As $v$ emits no edges outside of $\Lambda_n,$ $v^H=v-\sum_{e\in \so^{-1}(v)\cap \ra^{-1}(\Lambda_n-H)}ee^*$ and $\so^{-1}(v)\cap \ra^{-1}(\Lambda_n-H)$ is finite. We have that both $v^H$ and $\sum_{e\in \so^{-1}(v)\cap \ra^{-1}(\Lambda_n-H)}ee^*$ are in $I^{\gr}.$ So, $v\in I^{\gr}.$ 
  
It remains to show that $T^G\subseteq I^{\gr}.$ If $v^G\in T^G$ then $v\in T=S-G,$ so $v^H\in I\subseteq I^{\gr}.$ The set
$\so^{-1}(v)\cap \ra^{-1}(G-H)$ is finite and $v^G=v^H+\sum_{e\in \so^{-1}(v)\cap \ra^{-1}(G-H)}ee^*,$ and, as $\ra(e)\in G\subseteq I^{\gr}$ for $e\in e\in \so^{-1}(v)\cap \ra^{-1}(G-H),$ both $v^H$ and $\sum_{e\in \so^{-1}(v)\cap \ra^{-1}(G-H)}ee^*$ are in $I^{\gr}.$ So, $v^G\in I^{\gr}.$   
\end{proof}

Using Theorem \ref{theorem_graded_envelope}, we prove the result from the title of the paper. 

\begin{theorem}
If $I=I((H,S), C, P)$ is an ideal of $L_K(E)$ and $G=\ol{H\cup C^0}^S,$ then \[\ann_l(I)=\ann_r(I)=\ann(I)=\ann(I^{\gr})=I(G^\bot, B_{G^\bot}).\]
\label{theorem_annihilators}
\end{theorem}
\begin{proof} 
If $I$ and $G$ are as in the assumption of the theorem, then $I^{\gr}=I(G,S-G)$ by Theorem \ref{theorem_graded_envelope}.  The relation $\ann(I^{\gr})=I(G^\bot, B_{G^\bot})$ holds by  \cite[Proposition 3.5]{Lia_annihilators} and Remark \ref{remark_on_S_bot}.

Recall that $\ann_l(I), \ann_r(I), $ and $\ann(I)$ are graded ideals by \cite[Proposition 3.1]{Lia_annihilators}.
As $I\subseteq I^{\gr}$ and  $\ann_l(I^{\gr})=\ann_r(I^{\gr})=\ann(I^{\gr})$ (by \cite[Corollary 3.3]{Lia_annihilators}), we have  that 
\[\ann(I^{\gr})\subseteq \ann_l(I),\;\;\ann(I^{\gr})\subseteq \ann_r(I), \;\mbox{ and }\;\ann(I^{\gr})\subseteq \ann(I).\]
For the converse inclusions, we show that $\ann_r(I)\subseteq \ann(I^{\gr}).$ The inclusion $\ann_l(I)\subseteq \ann(I^{\gr})$ follows by symmetry of the proof and these two imply that 
$\ann(I)\subseteq \ann_l(I)\cap \ann_r(I)\subseteq \ann(I^{\gr}).$

If $G'=\ann_r(I)\cap E^0$ and $T'=\{v\in B_{G'}\mid v^{G'}\in \ann_r(I)\},$ then showing $G'\subseteq G^\bot$ and $T'\subseteq B_{G^\bot}$ is sufficient for $\ann_r(I)\subseteq \ann(I^{\gr}).$ 
Let $v\in G'$ so that $Iv=0.$ For $v\in G^\bot=E^0-R(G),$ we need to show that $v\notin R(G).$ Assume, on the contrary, that $v\in R(G)$ so that there is a path $p$ from $v$ to a vertex of $G.$ As $\ra(p)\in G\subseteq I^{\gr},$ we have that $pp^*=p\ra(p)p^*\in I^{\gr}.$ So, there is $a\in I$ such that $a_0=pp^*.$ Since $a\in I$ and $v\in \ann_r(I),$ $av=0.$ If $a=\sum_{n\in \Zset} a_n,$ then  $a_nv=0$ for any $n\in \Zset.$ In particular, $a_0v=0,$ so $pp^*=pp^*v=0.$ This is a contradiction since $0\neq p=pp^*p.$ Thus, $v\notin R(G)$ which shows that $G'\subseteq G^\bot.$
This inclusion and the inclusion $I(G^\bot, T^\bot)=\ann(I^{\gr})\subseteq \ann_r(I)=I(G', T')$ imply that $G'=G^\bot.$ Hence, $ T'\subseteq B_{G'}=B_{G^\bot}.$ 
\end{proof}

\section{Quasi-Baer Leavitt path algebras}
\label{section_quasi_baer}

\subsection{Graded quasi-Baer property} If $R$ is a $\Gamma$-graded ring, we say that $R$ is {\em graded quasi-Baer} if for any graded right ideal $I,$ there is a homogeneous idempotent $\varepsilon\in R$ such that $\ann_r(I)=\varepsilon R.$ Just as in the ungraded case, this definition is left-right symmetric and the condition that $I$ is one-sided can be replaced with the requirement that $I$ is double-sided in the definition. The proof of these claims is completely analogous to the proof of \cite[Lemma 1]{Clark}. 

Next, we relate the quasi-Baer and graded quasi-Baer conditions in Proposition \ref{proposition_graded_vs_nongraded_quasi_Baer}. We use the following two lemmas, most likely both well known, which we include for completeness. 

\begin{lemma}
If $R$ is a $\Gamma$-graded ring and $M$ a graded right $R$-module, then $\ann_r(M)$ is a graded ideal. Analogous claims hold for left modules and for bimodules.  
\label{lemma_annihilator_of_graded}
\end{lemma}
\begin{proof}
If $r\in \ann_r(M),$ let $r=\sum_{\gamma\in \Gamma} r_\gamma$ be such that $r_\gamma\in R_{\gamma}.$ Let $\delta\in \Gamma$ and $m\in M_\delta$ be arbitrary. As  $0=mr=\sum_{\gamma\in\Gamma} mr_{\gamma},$ $mr_{\gamma}=0$ for all $\gamma\in\Gamma,$ so $r_{\gamma}\in \ann_r(M)$ for any $\gamma\in \Gamma.$ Thus,  $\ann_r(M)$ is graded. The claims for left modules and bimodules are showed analogously. 
\end{proof}

\begin{lemma}
If $R$ is a unital $\Gamma$-graded ring and if $I$ is a graded right ideal for which there is a right ideal $J$ such that $I\oplus J=R_R,$ then $J$ is graded and there is a {\em homogeneous} idempotent $\varepsilon\in R$ such that $J=\varepsilon R.$ 
\label{lemma_graded_direct_summad}
\end{lemma}
\begin{proof}
If $I$ and $J$ are as in the assumption of the lemma, let $r=\sum_{\gamma\in\Gamma}r_\gamma\in J,$ let  $r_{\gamma}=a_\gamma+b_\gamma$ where $a_\gamma\in I$ and $b_\gamma\in J,$ and let  $a=\sum_{\gamma\in\Gamma} a_\gamma$ and $b=\sum_{\gamma\in\Gamma} b_\gamma.$ Since $a\in I,$  $b\in J,$ and $r=a+b\in J,$ we have that $a=0,$ so $a_\gamma=0$ for all $\gamma\in\Gamma.$ Hence,   
$r_{\gamma}=b_\gamma\in J$ which shows that $J$ is graded. Consequently, the short exact sequence
$0\to I\to R\to J\to 0$ 
of {\em graded} right $R$-modules is split, so there is a {\em graded} homomorphism   $\varepsilon'$ in the endomorphism ring of $R$ which is idempotent and such that $J=\varepsilon'(R).$ The idempotent $\varepsilon=\varepsilon'(1_R)$, where $1_R$ is the identity of $R$, is homogeneous and such that  $J=\varepsilon R.$ 
\end{proof}

\begin{proposition}
Let $R$ be any ring. 
\begin{enumerate}[\upshape(1)]
\item If $R$ is graded and quasi-Baer, then $R$ is graded quasi-Baer.
\item If $R$ is graded quasi-Baer and such that the annihilators of ideals are graded ideals, then $R$ is quasi-Baer.
\end{enumerate}   
\label{proposition_graded_vs_nongraded_quasi_Baer}
\end{proposition}
\begin{proof}
Note that the assumptions of both parts imply that $R$ is unital. In this case, we let $1_R$ denote the identity of $R$.

Let the assumption of (1) hold for $R$ and let $I$ be a graded right ideal. As $R$ is quasi-Baer, $\ann_r(I)$ is a direct summand of $R.$ By Lemma \ref{lemma_annihilator_of_graded}, $\ann_r(I)$ is a graded ideal of $R$. By Lemma \ref{lemma_graded_direct_summad}, there is a homogeneous idempotent $\varepsilon\in R$ such that $\ann_r(I)=\varepsilon R.$  

Let the assumptions of (2) hold for $R$ and let $I$ be an ideal. By the assumption, $\ann_r(I)$ is a graded ideal. As $R$ is graded quasi-Baer,  $\ann_l(\ann_r(I))=R\varepsilon$ for a homogeneous idempotent $\varepsilon\in R.$
So, $1_R-\varepsilon$ is homogeneous and $\ann_r(I)=\ann_r(\ann_l(\ann_r((I)))=\ann_r(R\varepsilon)=(1_R-\varepsilon)R.$ 
\end{proof}

The annihilators of ideals of a Leavitt path algebra are graded by \cite[Proposition 3.1]{Lia_annihilators} (also by Theorem \ref{theorem_annihilators}). Thus, by Proposition \ref{proposition_graded_vs_nongraded_quasi_Baer}, a Leavitt path algebra is quasi-Baer if and only if it is graded quasi-Baer.  

\subsection{Quasi-Baer Leavitt path algebras}
Recall that a graph $E$ and a field $K$ were fixed. By Theorem \ref{theorem_annihilators}, we can drop the subscripts $l$ and $r$ from $\ann_l(I)$ and $\ann_r(I)$ for an ideal $I$ of $L_K(E)$ and write only $\ann(I)$ without any danger of ambiguity. Let $\ann^2(I)$ shorten $\ann(\ann(I))$ and 
$\ann^3(I)$ shorten $\ann^2(\ann(I)).$ By  \cite[Propositions 3.7 and 3.10]{Lia_annihilators} (also by Theorem \ref{theorem_annihilators}),
$\ann^3(I)=\ann(I).$ 

The algebra $L=L_K(E)$ is a semiprime ring (if $I^2=0$ then $I=0$ for any ideal $I,$ see \cite[Proposition 2.3.1]{LPA_book}). So, for an ideal $I,$ if $a\in I\cap \ann(I)$ then $aLa=0.$ As $L$ is locally unital and semiprime,  $a=0,$ and so 
$I\cap \ann(I)=0.$
This implies that the lattice of annihilator ideals is a Boolean algebra where the meet is the intersection and the join is given by 
\[I\vee J=\ann(\ann(I)\cap \ann(J))=\ann^2(I+J)\]
(see also \cite[Section 3.6]{Lia_annihilators}). Thus, the join of $\ann(I)$ and $\ann^2(I)$ is the annihilator of $\ann^2(I)\cap \ann^3( I)=\ann^2(I)\cap \ann(I)=0,$ so 
$\ann(I)\vee \ann^2(I)=L.$
The annihilator $\ann(I)$ is a direct summand of $L$ exactly when this join is equal to the sum, i.e when $\ann(I)+\ann^2(I)=L.$

By \cite[Proposition 2.5.6]{LPA_book}, the smallest graded ideal which contains two graded ideals $I(H_1, S_1)$ and $I(H_2, S_2)$ is the graded ideal corresponding to the admissible pair 
\[(H_1, S_1)\vee (H_1, S_1)=(\ol{H_1\cup H_2}^{S_1\cup S_2}, S_1\cup S_2-\ol{H_1\cup H_2}^{S_1\cup S_2}).\]
To shorten the notation, if $S_1=B_{H_1}$ and $S_2=B_{H_2},$ let us 
denote the set $\ol{H_1\cup H_2}^{S_1\cup S_2}$ by $H_1\vee H_2.$ 

If $I$ is an ideal and  $I^{\gr}=I(G,T),$ the condition $\ann(I)+\ann^2(I)=L_K(E)$ 
is equivalent to 
the graph-theoretic condition $(G^\bot, B_{G^\bot})\vee (G^{\bot\bot}, B_{G^{\bot\bot}})=(E^0, \emptyset),$ which we write shorter as 
\[G^\bot\vee G^{\bot\bot}=E^0.\]

The following examples illustrate that the above condition does not necessarily  hold even for graded ideals $I$ which are closed in the sense that $\ann^2(I)=I.$  

\begin{example}
Let $E$ be any of the graphs below. 
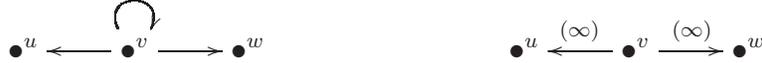
\[\xymatrix{\bullet^u&\bullet^v\ar[l]\ar[r]\ar@(ul, ur)&\bullet^w}\hskip3cm\xymatrix{\bullet^u&\bullet^v\ar[l]_{(\infty)}\ar[r]^{(\infty)}&\bullet^w}\]
and let $I=I(\{u\}, \emptyset).$ As $I$ is graded, $I=I^{\gr},$ so $H=G=\{u\},$ $G^\bot=E^0-R(G)=E^0-\{u,v\}=\{w\},$ and $G^{\bot\bot}=E^0-R(E^0-R(G))=E^0-\{v,w\}=\{u\}=G.$ We have that $B_{G^\bot}=\emptyset$ and $B_{G^{\bot\bot}}=B_G=\emptyset,$ so the set $G^\bot\vee G^{\bot\bot}$ is just the saturation of $G\cup G^\bot=\{u,w\}.$ The set $G\cup G^\bot$ is hereditary and saturated already, so this saturation is $\{u,w\}\subsetneq E^0.$ Note also that $\ann(I)=I(w),$ $\ann^2(I)=I=I(u),$ and $I(u)+I(w)=I(\{u,w\})$ does not contain $v.$ 
\label{example_not_quasi_Baer}
\end{example}

\begin{proposition}
The following conditions are equivalent.
\begin{enumerate}[\upshape(1)]
\item The algebra $L=L_K(E)$ is quasi-Baer.  

\item The set $E^0$ is finite and  $H^\bot\vee H^{\bot\bot}=E^0$ for any hereditary and saturated set $H.$
\end{enumerate}
\label{proposition_quasi_baer}
\end{proposition}
\begin{proof}
If (1) holds, then $L$ is unital, so $E^0$ is finite. If $H$ is any hereditary and saturated set, let $I=I(H,\emptyset)$ so that $\ann(I)=I(H^\bot, B_{H^\bot}).$ By (1), there is an idempotent $\varepsilon\in L$ such that $\ann(I)=\varepsilon L,$
so $J=(1_L-\varepsilon)L$ is a right ideal disjoint from $\ann(I)$ and such that $J+\ann(I)=L.$ As $L$ is semiprime, this implies that $J=\ann_r(\ann(I)),$ so $J$ is a double-sided ideal of $L$. Moreover, $J$ is graded by Lemma \ref{lemma_graded_direct_summad} and  $J=\ann^2(I)=I(H^{\bot\bot}, B_{H^{\bot\bot}}).$ As $\ann(I)\oplus J=L,$ $H^\bot\vee H^{\bot\bot}=E^0.$ 

To show the converse, assume that (2) holds, so $L$ is unital. Let $I$ be any ideal of $L.$ By Theorem \ref{theorem_annihilators}, $\ann(I)=I(G^\bot,B_{G^\bot})$ where $G$ is the set from the admissible pair $(G,T)$ such that $I^{\gr}=I(G,T).$ The relation $G^\bot\vee G^{\bot\bot}=E^0$ implies that $\ann(I)+\ann^2(I)=I(E^0, \emptyset)=L.$ As $L$ is semiprime, $\ann(I)\cap \ann^2(I)=0,$ so $\ann(I)\oplus\ann^2(I)=L.$ Thus, (1) holds.  
\end{proof}

\subsection{Relation to Baer and Rickart Leavitt path algebras}
Recall that a ring $R$ is {\em Baer} if, for any set $X\subseteq R,$ there is an idempotent $\varepsilon\in R$ such that $\ann_r(X)=\varepsilon R$ and that $R$ is {\em right Rickart} if the same condition holds when $|X|=1.$ 
Left Rickart ring is defined by an analogous condition for left annihilators and a ring is {\em Rickart} if it is both left and right Rickart.
Each of these conditions implies that $R$ is unital.

It is direct that a Baer ring is both Rickart and quasi-Baer and it is known that both implications are strict. As all three conditions have now been characterized for Leavitt path algebras, it is rather direct to obtain examples of algebras which are Rickart or quasi-Baer and not Baer. By \cite[Proposition 13 and Theorem 15]{Roozbeh_Lia_Baer}, $L_K(E)$ is Baer if and only if $E$ is a finite graph in which no cycle has exits and $L_K(E)$ is Rickart if and only if $E^0$ is finite. Thus, if $E$ is the  graph $\;\;\;\;\xymatrix{\bullet\ar@(dr, ur)\ar@(dl, ul)}\;\;\;\;,$ then $L_K(E)$ is Rickart and not Baer. There are no nonempty and proper hereditary sets, 
so condition (2) of Proposition \ref{proposition_quasi_baer} trivially holds and, hence, $L_K(E)$ is quasi-Baer by Proposition \ref{proposition_quasi_baer}. If $E$ is any of the graphs from Example \ref{example_not_quasi_Baer}, then $L_K(E)$ is Rickart and not quasi-Baer. As any unital Leavitt path algebra is Rickart (by \cite[Proposition 13]{Roozbeh_Lia_Baer}), a quasi-Baer Leavitt path algebra is Rickart. 

\subsection{Quasi-Baer *-rings} Recall that for any  $a\in L_K(E),$ $a=\sum_{i=1}^n k_ip_iq_i^\ast$ for some $n$, paths $p_i$ and $q_i$, and $k_i\in K,$ for $i=1,\ldots,n$ where $v^*=v$ for $v\in E^0$ and $p^*=e_n^*\ldots e_1^*$ for a path $p=e_1\ldots e_n.$ If $k_i\mapsto k_i^*$ is any involution on $K$, then letting  
$\left(\sum_{i=1}^n k_ip_iq_i^\ast\right)^*=\sum_{i=1}^n k_i^*q_ip_i^\ast$ gives  
$L_K(E)$ the structure of an involutive $K$-algebra.  
The involutive properties of $L_K(E)$ are fundamentally  impacted by the properties of the involution on $K.$ In particular, the involution on $K$ is positive definite (i.e for any $n$ and any $k_i\in K, i=1,\ldots, n$ if $\sum_{i=1}^n k_ik_i^*=0$ then $k_i=0$ for all $i=1,\ldots n$) if and only if the involution on $L_K(E)$ is positive definite (see \cite[Proposition 12]{Roozbeh_Lia_Baer}). The involution on $L_K(E)$ is compatible with the natural grading in the sense that  $L_K(E)_n^*\subseteq L_K(E)_{-n}$ for any $n\in\Zset.$

In an involutive ring, a projection is a self-adjoint idempotent. If ``idempotent'' is replaced by ``projection'' in the definitions of a Baer, a quasi-Baer, and a Rickart ring, we obtain the definitions of a Baer $*$-ring, a quasi-Baer $*$-ring, and a Rickart $*$-ring, respectively.
The graded version of a quasi-Baer $*$-ring is obtained by requiring the ideal to be graded and the projection to be homogeneous. 

If the involution on $K$ is positive definite, \cite[Proposition 13 and Theorems 15 and 16]{Roozbeh_Lia_Baer} characterize Leavitt path algebras which are Baer $*$-rings, graded Baer $*$-rings and graded Rickart $*$-rings. These results and examples from \cite{Roozbeh_Lia_Baer} show that while Baer, graded Baer and graded Baer $*$ are equivalent conditions, Baer $*$ is strictly stronger from these and the same statement holds if ``Baer'' is replaced with ``Rickart''. In contrast, all four quasi-Baer conditions are equivalent for  Leavitt path algebras as we show next. 
 
\begin{proposition} If $K$ is a field with positive definite involution, the following conditions are equivalent.  
\begin{enumerate}[\upshape(1)]
\item The algebra $L_K(E)$ is a  quasi-Baer $*$-ring.

\item The algebra $L_K(E)$ is a graded quasi-Baer $*$-ring. 

\item The algebra $L_K(E)$ is graded quasi-Baer.

\item The algebra $L_K(E)$ is quasi-Baer. 
\end{enumerate}
\label{proposition_star_qB}
\end{proposition}
\begin{proof}
The implication (2) $\Rightarrow$ (3) is direct and (3)$\Leftrightarrow$(4) holds by Proposition  \ref{proposition_graded_vs_nongraded_quasi_Baer} (see the paragraph following the proof of  Proposition  \ref{proposition_graded_vs_nongraded_quasi_Baer}). We show the implications (1) $\Rightarrow$ (2) and (4) $\Rightarrow$ (1).

If (1) holds and $I$ is a graded ideal of $L=L_K(E),$ the proof of (1) $\Rightarrow$ (2) of Proposition \ref{proposition_graded_vs_nongraded_quasi_Baer} shows that there is a homogeneous idempotent $\varepsilon$ such that $\ann(I)=\varepsilon L.$ As $L$ is graded regular (by \cite[Theorem 9]{Roozbeh_regular}) and the involution on $L$ is positive definite (by \cite[Proposition 12]{Roozbeh_Lia_Baer}), for every homogeneous idempotent $\varepsilon$ of $L,$ there is a homogeneous projection $\pi$ such that $\varepsilon L=\pi L$ (by \cite[Proposition 5]{Roozbeh_Lia_Baer}). Thus, $\ann(I)=\pi L$ for a homogeneous projection $\pi.$ So, (2) holds. 

To show (4)$\Rightarrow$(1), assume that (4) holds. By \cite[Proposition 1.1]{Birkenmeier_et_al} and as $L=L_K(E)$ is semiprime, to show (1) it is sufficient to show that each central idempotent is a projection. If $\varepsilon$ is a central idempotent, then $\varepsilon L=L\varepsilon$ is a double-sided ideal and so $\ann(\varepsilon L)=\ann_r(L\varepsilon )=(1_L-\varepsilon)L$ is a graded ideal. So,  
$\ann((1_L-\varepsilon) L)=\varepsilon L$ is also a graded ideal by Lemma \ref{lemma_graded_direct_summad}. As every graded ideal is self-adjoint (see \cite[Corollary 2.4.10]{LPA_book}), $\varepsilon L=L\varepsilon^*.$ This relation and $\varepsilon$ being central, imply that 
$\varepsilon L=\varepsilon^*L,$ so that $\varepsilon=\varepsilon^*\varepsilon=\varepsilon\varepsilon^*=\varepsilon^*.$ Thus, $\varepsilon$ is a projection. 
\end{proof}

\end{document}